\documentclass[11pt,reqno]{amsart}

\usepackage{amssymb,amsmath,amsfonts,amsthm}
\usepackage{bm,cite,graphicx,color}
\usepackage[toc,page]{appendix}

\newtheorem{theorem}{Theorem}[section]
\newtheorem{lemma}[theorem]{Lemma}

\newtheorem{definition}[theorem]{Definition}
\newtheorem{remark}[theorem]{Remark}
\newtheorem{assumption}{Assumption}
\numberwithin{equation}{section}

\allowdisplaybreaks[4]

\begin{document}

\title[stochastic acoustic and elastic scattering problems]{Regularity of
distributional solutions to stochastic acoustic and elastic scattering problems}

\author{Peijun Li}
\address{Department of Mathematics, Purdue University, West Lafayette, Indiana
47907, USA.}
\email{lipeijun@math.purdue.edu}

\author{Xu Wang}
\address{Department of Mathematics, Purdue University, West Lafayette, Indiana
47907, USA.}
\email{wang4191@purdue.edu}

\thanks{The research is supported in part by the NSF grant DMS-1912704.}

\subjclass[2010]{35R60, 60H15, 35B65}

\keywords{regularity, the Helmholtz equation, the elastic wave equation, random
media, microlocally isotropic generalized Gaussian random field} 

\begin{abstract}
This paper is concerned with the well-posedness and regularity of the
distributional solutions for the stochastic acoustic and elastic scattering
problems. We show that the regularity of the solutions depends on the
regularity of both the random medium and the random source. 
\end{abstract}

\maketitle
\section{Introduction}

The acoustic and elastic wave equations are two fundamental equations
to describe wave propagation. They have significantly applications in diverse
scientific
areas such as remote sensing, nondestructive testing, geophysical prospecting,
and medical imaging \cite{CK13}. In practice, due to the unpredictability of the
environments and incomplete knowledge of the systems, the radiating sources
and/or the host media, and hence the radiated fields may not be deterministic
but rather are modeled by random fields \cite{FGPS07}. Their governing
equations are some forms of stochastic differential equations and their
solutions are random fields instead of their deterministic counterparts of
regular functions \cite{C75, I78, K62}. Regularity theory of stochastic wave
equations has played an important role in the study of partial differential
equations and attracted a lot of attention \cite{J02, DF98, MS99}. As is known,
a basic problem in classical scattering theory is the scattering of a
time-harmonic wave by an inhomogeneous medium. This paper is concerned with the
well-posedness and regularity of the solutions for the time-harmonic stochastic
acoustic and elastic scattering problems.

For the case of acoustic waves, it is to find the induced pressure $u$ which
satisfies the Helmholtz equation
\begin{align}\label{eq:H}
\Delta u+k^2(1+\rho)u=f\quad\text{in}~\mathbb R^d,
\end{align}
where $d=2$ or $3$, $k>0$ is the wavenumber, $\rho$ describes the inhomogeneous
medium and is assumed to be a microlocally isotropic generalized Gaussian random
field (cf. Definition \ref{def:iso}) defined in a bounded domain $D_\rho$, and
$f$ is assumed to be either a microlocally isotropic generalized Gaussian random
field in a bounded domain $D_f$ or a point source given by a delta
distribution. In addition, the pressure $u$ is required to satisfy the Sommerfeld radiation condition 
\begin{align}\label{eq:RCH}
\lim_{|x|\to\infty}|x|^{\frac{d-1}2}\left(\partial_{|x|} u-{\rm
i}ku\right)=0.
\end{align}

The elastic analogue is to find the displacement $\bm u$ satisfying the Navier
equation
\begin{align}\label{eq:E}
\mu\Delta\bm u+(\lambda+\mu)\nabla\nabla\cdot\bm u+k^2(\bm I+\bm M)\bm u=\bm f\quad\text{in}~\mathbb R^d,
\end{align}
where $\bm I$ is the identity matrix in $\mathbb R^d$, the
Lam\'e parameters $\mu$ and $\lambda$ satisfy $\mu>0$ and $\lambda+2\mu>0$ such
that the linear operator $\Delta^*:=\mu\Delta+(\lambda+\mu)\nabla\nabla\cdot$ is
uniformly elliptic (cf. \cite[(10.4)]{M00}), ${\bm M}$ represents the
anisotropic, inhomogeneous medium and is assumed to be a $\mathbb R^{d\times
d}$-valued microlocally isotropic generalized Gaussian random field in a
bounded domain $D_{\bm M}$, and $\bm f$ is either a microlocally isotropic
generalized Gaussian random field in a bounded domain $D_{\bm f}$ or a point
source given by a delta distribution. By \cite{BLZ20}, the displacement admits
the Helmholtz decomposition 
\[
 \bm u = \bm u_{\rm p} + \bm u_{\rm s}\quad \text{in}\,\mathbb R^d\setminus
(\overline{D_{\bm M}\cup D_{\bm f}}),
\]
where $\bm u_{\rm p}$ and $\bm u_{\rm s}$ are the compressional and shear wave
components, respectively, and are required to satisfy the  Kupradze--Sommerfeld
radiation condition 
\begin{eqnarray}\label{eq:RCE}
&&\lim_{|x|\to\infty}|x|^{\frac{d-1}2}\left(\partial_{|x|}\bm u_{\rm
p}-{\rm i}\kappa_{\rm p}\bm u_{\rm p}\right)
=\lim_{|x|\to\infty}|x|^{\frac{d-1}2}\left(\partial_{|x|}\bm
u_{\rm s}-{\rm i}\kappa_{\rm s}\bm u_{\rm s}\right)=0,\qquad\quad
\end{eqnarray}
where
\[
 \kappa_{\rm p}=k/(\lambda+2\mu)^{1/2},\quad \kappa_{\rm s}=k/\mu^{1/2}
\]
are called the compressional and shear wavenumbers, respectively.

Recently, the microlocally isotropic generalized Gaussian random fields are
adopted to characterize the random coefficients of some stochastic wave
equations. The associated covariance operators can be viewed as classical
pseudo-differential operators. These random fields may be too rough to be
classical functions, and should be interpreted as distributions instead. Classical regularity estimates are not
applicable for these stochastic equations due to the roughness of the random
coefficients. The well-posedness of these equations in the distribution
sense and regularity of the distributional solutions need to be investigated.
We refer to \cite{CHL19,LLW} and \cite{LHL,LL19} for the study of the
well-posedness of the solutions for the acoustic and elastic wave equations
with random potentials and sources, respectively.
However, it remains open for the well-posedness and regularity of the solutions
for the stochastic acoustic and elastic wave scattering problems in random media
with random sources. The goal of this paper is to examine the well-posedness and
regularity of the distributional solutions for the stochastic acoustic
scattering problem \eqref{eq:H}--\eqref{eq:RCH} and the stochastic elastic
scattering problem \eqref{eq:E}--\eqref{eq:RCE} by using a unified approach.

This paper is organized as follows. In Section 2, we introduce some Sobolev
spaces of real order and the microlocally isotropic generalized Gaussian
random fields. Sections 3 and 4 address the well-posedness and regularity of the
solutions for the stochastic acoustic and elastic scattering problems,
respectively.

\section{Preliminaries}

In this section, we briefly introduce Sobolev spaces of real order
and microlocally isotropic generalized Gaussian random fields, which are
used in this paper. 

\subsection{Sobolev spaces}

Let $C_0^{\infty}(D)$ be the set of smooth functions compactly
supported in $D\subset\mathbb R^d$, and $\mathcal{D}(D)$ be the space of test
functions, which is $C_0^{\infty}(D)$ equipped with a locally convex
topology  (cf. \cite{AF03}). The dual space $\mathcal{D'}(D)$ of
$\mathcal{D}(D)$ is called the space of distributions on $D$ equipped with a
weak-star topology. Define the product 
\[
\langle u,v\rangle:=\int_{D}u(x)\overline{v(x)}dx
\]
for $u\in\mathcal{D}'(D)$ and $v\in\mathcal{D}(D)$. The distributional partial derivative of $u\in\mathcal{D}'(D)$ satisfies
\[
\langle \partial^\zeta u,\psi\rangle=(-1)^{|\zeta|}\langle u,\partial^\zeta\psi\rangle
\]
for any $\psi\in\mathcal{D}(D)$ and multi-index $\zeta=(\zeta_1, \dots, \zeta_d)$.

For any positive integer $n$ and $1\le p<\infty$, the Sobolev space $W^{n,p}(D)$ is defined by
\[
W^{n,p}(D)=\{u\in L^p(D): \partial^\zeta u\in
L^p(D)\quad\text{for}\quad0\le|\zeta|\le n\},
\]
which is equipped with the norm
\[
\|u\|_{W^{n,p}(D)}:=\left(\sum_{0\le|\zeta|\le n}\|\partial^\zeta u\|_{L^p(D)}^p\right)^{\frac1p}.
\]

For any $r\in\mathbb R_+$, let $r=n+\mu$ with $n=[r]$ being the largest
integer smaller than $r$ and $\mu\in(0,1)$, and define
\[
W^{r,p}(D)=\{u\in W^{n,p}(D): |\partial^\zeta u|_{W^{\mu,p}(D)}<\infty\quad\text{for}\quad |\zeta|=n\}
\]
equipped with the norm
\[
\|u\|_{W^{r,p}(D)}:=\left(\|u\|_{W^{n,p}(D)}^p+\sum_{|\zeta|=n}|\partial^\zeta u|_{W^{\mu,p}(D)}^p\right)^{\frac1p},
\]
where
\[
|u|_{W^{\mu,p}(D)}:=\left(\int_D\int_D\frac{|u(x)-u(y)|^p}{|x-y|^{p\mu+d}}dxdy\right)^{\frac1p}
\]
is the Slobodeckij semi-norm. Denote by $W_0^{r,p}(D)$ the closure of
$C_0^{\infty}(D)$ in $W^{r,p}(D)$.

For any $r\in\mathbb R_+$, the Sobolev space $W^{-r,p}(D)$ of negative order is defined as the dual of $W^{r,q}_0(D)$ with $\frac1p+\frac1q=1$ equipped with the norm
\[
\|u\|_{W^{-r,p}(D)}:=\sup_{v\in W^{r,q}(D),\|v\|_{W^{r,q}(D)}\le1}|\langle u,v\rangle|.
\]

If $D=\mathbb R^d$, there is another kind of Sobolev spaces defined through the
Bessel potential. Let $\mathcal{S}(\mathbb R^d)$ be the Schwartz space of
rapidly decreasing smooth functions, i.e.,
\[
\mathcal{S}(\mathbb R^d):=\{\phi\in C^\infty(\mathbb R^d):\sup_{x\in\mathbb R^d}|x^\zeta\partial^\tau\phi(x)|<\infty\quad\text{for all multi-indices $\zeta$ and $\tau$}\},
\]
and $\mathcal{S}'(\mathbb R^d)$ be the dual space of $\mathcal{S}(\mathbb
R^d)$. Then $\mathcal{D}(\mathbb R^d)\subset\mathcal{S}(\mathbb R^d)$ and
$\mathcal{S}'(\mathbb R^d)\subset\mathcal{D}'(\mathbb R^d)$. For any
$s\in\mathbb R$, define the Bessel potential $\mathcal{J}^s:\mathcal{S}(\mathbb
R^d)\to\mathcal{S}(\mathbb R^d)$ of order $s$ by
\[
\mathcal{J}^su:=(I-\Delta)^{\frac s2}u=\mathcal{F}^{-1}[(1+|\cdot|^2)^{\frac
s2}\hat u],
\]
where $\mathcal{F}^{-1}$ is the inverse Fourier transform. It is easy to verify
that 
\[
(\mathcal{J}^su,v)_{L^2(\mathbb R^d)}=(u,\mathcal{J}^sv)_{L^2(\mathbb R^d)}
\quad \forall\,u,v\in\mathcal{S}(\mathbb R^d),
\] 
where $(\cdot,\cdot)_{L^2(\mathbb R^d)}$ is the inner product in
$L^2(\mathbb R^d)$ satisfying
\[
(u,v)_{L^2(\mathbb R^d)}=\langle u,v\rangle.
\]  
Based on the Bessel potential, we introduce the following Sobolev space of order $s\in\mathbb R$: 
\[
H^{s,p}(\mathbb R^d)=\{u\in\mathcal{S}'(\mathbb R^d):\mathcal{J}^su\in L^p(\mathbb R^d)\}.
\]
Denote $H^s(\mathbb R^d):=H^{s,2}(\mathbb R^d)$, which is a Hilbert space with the inner product
\[
(u,v)_{H^s(\mathbb R^d)}:=(\mathcal{J}^su,\mathcal{J}^sv)_{L^2(\mathbb R^d)}
\]
and the induced norm
\[
\|u\|_{H^s(\mathbb R^d)}:=\|\mathcal{J}^su\|_{L^2(\mathbb R^d)}.
\]

For any set $D\subset\mathbb R^d$, define
\[
H^s(D)=\{u\in\mathcal{D}'(D):u=\tilde u|_D\quad\text{for some extension }\tilde u\in H^s(\mathbb R^d)\}.
\]
Then it holds $H^s(D)=W^{s,2}(D)$ for any real $s\ge0$ and $H^{-n}(D)=W^{-n,2}(D)$ for any integer $n\ge0$ with equivalent norms (cf. \cite{M00,G11}).

\subsection{Microlocally isotropic generalized Gaussian random fields}

Denote by $(\Omega, \mathcal{F},\mathbb{P})$ a complete probability space, where $\Omega$ is a sample space, $\mathcal F$ is a $\sigma$-algebra on $\Omega$, and $\mathbb P$ is a probability measure on the measurable space $(\Omega, \mathcal F)$. Define $\mathcal{D}:=\mathcal{D}(\mathbb{R}^d)$ and
$\mathcal{D'}:=\mathcal{D'}(\mathbb R^d)$. 
A real-valued field $\rho$ is said to be a generalized random field if, for 
each $\omega\in\Omega$, the realization $\rho(\omega)$ belongs to $\mathcal{D'}$ and the mapping 
\begin{eqnarray}\label{l1}
\omega\in\Omega\longmapsto \langle \rho(\omega),\psi\rangle\in\mathbb R
\end{eqnarray}
is a random variable for all $\psi\in \mathcal{D}$.

In particular, a generalized random field is said to be Gaussian if (\ref{l1}) defines a Gaussian random variable for all $\psi\in \mathcal{D}$.
A generalized Gaussian random field $\rho\in\mathcal{D}'$ is uniquely determined by its expectation $\mathbb E\rho\in\mathcal{D}'$ and 
covariance operator $Q_{\rho}: \mathcal{D}\rightarrow \mathcal{D'}$ defined by
\begin{eqnarray*}
\langle\mathbb{E}\rho,\psi\rangle&:=&\mathbb{E}\langle \rho,\psi\rangle\quad\forall\,\psi\in\mathcal{D},\\
\langle Q_{\rho}\psi_1,\psi_2\rangle&:=&\mathbb E\left[(\langle
\rho,\psi_1\rangle-\mathbb E\langle \rho,\psi_1\rangle)(\langle
\rho,\psi_2\rangle-\mathbb E\langle \rho,\psi_2\rangle)\right]\quad\forall\,\psi_1,\psi_2\in\mathcal{D}.
\end{eqnarray*}

It follows from the continuity of $Q_\rho$ and the Schwartz kernel theorem that there exists a unique kernel function $\mathcal{K}_\rho(x,y)$ satisfying 
\begin{eqnarray*}\label{l3}
\langle  {Q}_{\rho}\psi_1,\psi_2\rangle=\int_{\mathbb{R}^d}\int_{\mathbb{R}^d}\mathcal{K}_\rho(x,y)\psi_1(x)\overline{\psi_2(y)}dxdy\quad\forall\,\psi_1,\psi_2\in \mathcal{D}.
\end{eqnarray*} 
The regularity of the covariance operator $Q_\rho$ determines the regularity of the random field $\rho$.

\begin{definition}\label{def:iso}
A generalized Gaussian random field $\rho$ on $\mathbb R^d$ is called
microlocally isotropic of order $-m$ with $m\ge0$ in $D$ if its covariance
operator $Q_\rho$ is a classical pseudo-differential operator having an
isotropic principal symbol $\phi(x)|\xi|^{-m}$ with the micro-correlation strength $\phi\in
C_0^\infty(D)$ being compactly supported in $D$ and $\phi\ge0$.
\end{definition}

Note that the covariance operator with a principle symbol $\phi(x)|\xi|^{-m}$ has similar regularity as the fractional Laplacian. To investigate the regularity of microlocally isotropic Gaussian random fields defined above,
we introduce the centered fractional Gaussian fields (cf. \cite{LSSW16, LW})
defined by
\begin{align}\label{eq:FGF}
h_m(x):=(-\Delta)^{-\frac m4}\dot{W}(x),\quad x\in\mathbb{R}^d,
\end{align}
where $(-\Delta)^{-\frac m4}$ is the fractional Laplacian and
$\dot{W}\in\mathcal{D}'$ denotes the white noise. It is shown in \cite{LW} that $h_m$ is a microlocally isotropic Gaussian random field of order
$-m$ satisfying Definition \ref{def:iso} with $\phi\equiv1$.
Hence the fractional Gaussian field $h_m$ defined by \eqref{eq:FGF} has
the same regularity as the microlocally isotropic Gaussian random field $\rho$
of order $-m$ in Definition \ref{def:iso}. 

In particular, if $m\in(d,d+2)$, the fractional Gaussian field $h_m$ defined above 
is a translation of a classical fractional Brownian motion. More precisely, 
\[
\tilde h_m(x):=\langle h_m,\delta(x-\cdot)-\delta(\cdot)\rangle,\quad x\in\mathbb{R}^d
\]
has the same distribution as the classical fractional Brownian motion with Hurst
parameter $H=\frac{m-d}2\in(0,1)$ up to a multiplicative constant,
where $\delta(\cdot)$ is the Dirac function centered at the origin.

Taking advantages of the relationship between the microlocally isotropic Gaussian
random fields and the fractional Gaussian fields defined in \eqref{eq:FGF}, we
conclude this section by providing the regularity of
microlocally isotropic Gaussian random fields, whose proof can be
found in \cite{LW}.

\begin{lemma}\label{lm:iso}
Let $\rho$ be a microlocally isotropic Gaussian random field of
order $-m$ in $D$ with $m\in[0,d+2)$. 

\begin{itemize}

\item[(i)] If $m\in(d,d+2)$, then $\rho\in C^{0,\alpha}(D)$ almost surely for
all $\alpha\in(0,\frac{m-d}2)$.

\item[(ii)] If $m\in[0,d]$, then $\rho\in W^{\frac{m-d}2-\epsilon,p}(D)$ almost
surely for any $\epsilon>0$ and $p\in(1,\infty)$.
\end{itemize}
\end{lemma}

\begin{remark}
For a microlocally isotropic Gaussian random field $\rho$ in Definition
\ref{def:iso}, its kernel has the form
$\mathcal{K}_\rho(x,y)=\phi(x)\mathcal{K}_{h_m}(x,y)+r(x,y)$, where
$\phi\mathcal{K}_{h_m}$ is the leading term with strength $\phi$ and
$r$ is a smooth residual (cf. \cite{LPS08}).
\end{remark}

\section{The acoustic scattering problem}

In this section, we consider the Helmholtz equation \eqref{eq:H}
and study the well-posedness for the acoustic scattering problem under the
following assumptions on the medium $\rho$ and the source $f$.

\begin{assumption}\label{as:rho}
Let the medium $\rho$ be a real-valued centered microlocally isotropic Gaussian
random field of order $-m_\rho$ with $m_\rho\in(d-1,d]$ in a bounded domain $D_\rho\subset\mathbb R^d$.
The principal symbol of its
covariance operator has the form $\phi_\rho(x)|\xi|^{-m_\rho}$ with
$\phi_\rho\in C_0^\infty(D_\rho)$ and $\phi_\rho\ge0$.
\end{assumption}

\begin{assumption}\label{as:f}
Let the real-valued source $f$ satisfy one of the following assumptions:
\begin{itemize}
\item[(i)] $f$ is a centered microlocally isotropic Gaussian random field of
order $-m_f$ with $m_f\in(d-1,d]$ in a bounded domain $D_f\subset\mathbb R^d$. The principal symbol
of its covariance operator has the form $\phi_f(x)|\xi|^{-m_f}$ with $\phi_f\in
C_0^\infty(D_f)$ and $\phi_f\ge0$.

\item[(ii)] $f=-\delta(\cdot-y)a$ is a point source with $y\in\mathbb R^d$ and
some fixed constant $a\in\mathbb R$.
\end{itemize} 
\end{assumption}

For such rough $\rho$ and $f$, the Helmholtz equation \eqref{eq:H} should be
interpreted in the distribution sense.  First let us consider the
equivalent Lippmann--Schwinger integral equation.

\subsection{The Lippmann--Schwinger equation}

Based on the fundamental solution
\begin{equation}\label{eq:Green}
\Phi_d(x,y,k)=\left\{
\begin{aligned}
\frac{\rm i}4H_0^{(1)}(k|x-y|),\quad d=2,\\
\frac{e^{{\rm i}k|x-y|}}{4\pi|x-y|},\quad d=3,
\end{aligned}
\right.
\end{equation} 
of the equation $\Delta u+k^2u=-\delta(\cdot-y)$ in $\mathbb R^d$, 
the Lippmann--Schwinger integral equation has the form
\begin{eqnarray}\label{eq:LS}
 u(x)-k^2\int_{\mathbb R^d} \Phi_d(x,z,k)\rho(z) u(z)dz= -\int_{\mathbb R^d}\Phi_d(x,z,k)f(z)dz.
\end{eqnarray}

Define two operators 
\begin{align*}
(H_k  v)(x)&:=\int_{\mathbb{R}^d} \Phi_d(x,z,k) v(z)dz,\\
(K_k  v)(x)&:=\int_{\mathbb{R}^d} \Phi_d(x,z,k)\rho(z) v(z)dz,
\end{align*}
which have the following properties.

\begin{lemma}\label{lm:operators}
Let $\rho$ satisfy Assumption \ref{as:rho}. Let $ D\subset\mathbb{R}^d$ be a bounded set and $
G\subset\mathbb{R}^d$ be a bounded set with a locally Lipschitz boundary.
\begin{itemize}
\item[(i)] The operator $H_k:H_0^{-\beta}( D)\to H^\beta( G)$ is bounded for any $\beta\in(0,1]$.

\item[(ii)] The operator $H_k: W_0^{-\gamma,p}( D)\to W^{\gamma,q}( G)$ is compact for any $q\in(2,\infty)$, $\gamma\in(0,(\frac1q-\frac12)d+1)$ and $p$ satisfying $\frac1p+\frac1q=1$.

\item[(iii)] The operator $K_k: W^{\gamma,q}( G)\to W^{\gamma,q}( G)$ is compact for any $q\in(2,\frac{2d}{2d-2-m_\rho})$ and $\gamma\in(\frac{d-m_\rho}2,(\frac1q-\frac12)d+1)$.
\end{itemize}
\end{lemma}

\begin{proof}
(i) It follows from \cite[Theorem 8.1]{CK13} that $H_k$ is bounded from $C^{0,\alpha}(D)$ to $C^{2,\alpha}(G)$ with respect to the corresponding H\"older norms $\|\cdot\|_{C^{0,\alpha}(D)}$ and $\|\cdot\|_{C^{2,\alpha}(G)}$. Define spaces $X:=C^{0,\alpha}(D)$ and $Y:=C^{2,\alpha}(G)$ with scalar products
\[
(f_1,f_2)_X:=(\tilde f_1,\tilde
f_2)_{H^{\beta-2}(\mathbb{R}^d)}\quad\forall\,f_1, f_2\in X
\]
and
\[
(g_1,g_2)_Y:=(\tilde g_1,\tilde g_2)_{H^{\beta}(\mathbb{R}^d)}\quad \forall\,
g_1, g_2\in Y,
\]
respectively, where $\tilde f_i$ and $\tilde g_i$ are the zero extensions of
$f_i$ and $g_i$  in $\mathbb R^d\setminus\overline D$ and $\mathbb
R^d\setminus\overline G$, respectively. It is easy to verify that 
the products defined above satisfy 
\begin{align*}
(f_1,f_2)_X &= (J^{\beta-2}\tilde f_1,J^{\beta-2}\tilde
f_2)_{L^2(\mathbb{R}^d)}=\int_{\mathbb R^d}(1+|\xi|^2)^{\beta-2}\hat{\tilde
f}_1(\xi)\overline{\hat{\tilde f}_2(\xi)}d\xi\\
&\lesssim \|\tilde f_1\|_{L^2(\mathbb R^d)}\|\tilde f_2\|_{L^2(\mathbb R^d)}
\lesssim\|f_1\|_{C^{0,\alpha}(D)}\|f_2\|_{C^{0,\alpha}(D)}
\end{align*}
and
\[
(g_1,g_2)_Y=(J^{\beta}\tilde g_1,J^{\beta}\tilde g_2)_{L^2(\mathbb{R}^d)}
\lesssim\|\tilde g_1\|_{H^\beta(\mathbb R^d)}\|\tilde g_2\|_{H^\beta(\mathbb R^d)}
\lesssim\|g_1\|_{C^{2,\alpha}(G)}\|g_2\|_{C^{2,\alpha}(G)},
\]
where the notation $a\lesssim b$ denotes $a\le Cb$ for some constant $C>0$.

We claim that there exists a bounded operator $V:Y\to X$ defined by $V=(I-\Delta)H_k(I-\Delta)$ such that
\[
(H_kf,g)_Y=(f,Vg)_X\quad\forall\,f\in X, g\in Y.
\]
In fact, for any $g\in Y$, 
\begin{align*}
\|Vg\|_{C^{0,\alpha}(D)}&=\|(I-\Delta)H_k(I-\Delta)g\|_{C^{0,\alpha}(D)}
\lesssim\|H_k(I-\Delta)g\|_{C^{2,\alpha}(D)}\\
&\lesssim \|(I-\Delta)g\|_{C^{0,\alpha}(G)}\lesssim\|g\|_{C^{2,\alpha}(G)}.
\end{align*}
Furthermore,
\begin{align*}
(H_kf,g)_Y&=(J^\beta H_k\tilde f,J^\beta\tilde g)_{L^2(\mathbb R^d)}=(H_k\tilde
f,J^{2\beta}\tilde g)_{L^2(\mathbb R^d)}\\
&=(\widehat{H_k\tilde f},\widehat{J^\beta\tilde g})_{L^2(\mathbb R^d)}
=\int_{\mathbb R^d}\hat\Phi_d(\xi)\hat{\tilde f}(\xi)(1+|\xi|^2)^\beta\overline{\hat{\tilde g}(\xi)}d\xi\\
&=\int_{\mathbb R^d}\hat{\tilde
f}(\xi)(1+|\xi|^2)^{\beta-2}\overline{\left[
(1+|\xi|^2)\hat\Phi_d(\xi)(1+|\xi|^2)\hat{\tilde g}(\xi)\right]}d\xi\\
&=\int_{\mathbb R^d}\hat{\tilde
f}(\xi)(1+|\xi|^2)^{\beta-2}\overline{\widehat{V\tilde g}(\xi)}d\xi
=(J^{\beta-2}\tilde f,J^{\beta-2}V\tilde g)_{L^2(\mathbb R^d)}\\
&=(f,Vg)_X,
\end{align*}
where $\hat\Phi_d$ is the Fourier transform of $\Phi_d(x,y,k)$ with respect to
$x-y$ and satisfies $-|\xi|^2\hat\Phi_d(\xi)+k^2\hat\Phi_d(\xi)=-1$. The claim
is proved. 

It follows from the claim and  \cite[Theorem 3.5]{CK13} that $H_k:X\to Y$ is
bounded with respect to the norms induced by the scalar products on $X$ and
$Y$. More precisely, we have 
\begin{align}\label{eq:Hk}
\|H_kf\|_Y=\|H_kf\|_{H^\beta(G)}\lesssim\|f\|_X=\|f\|_{H^{\beta-2}(D)}\le\|f\|_{H^{-\beta}(D)}
\end{align}
for any $f\in X$ and $\beta\le1$. It then suffices to show that \eqref{eq:Hk} also holds for any $f\in H^{-\beta}(D)$.
Noting that the subspace $C_0^\infty(D)\subset X$ is dense in $L^2(D)$ (cf. \cite[Section 2.30]{AF03}) and 
$
H^{-1}(D)=\overline{L^2(D)}^{\|\cdot\|_{H^{-1}(D)}}
$
(cf. \cite[Section 3.13]{AF03}), we get that \eqref{eq:Hk} holds for any $f\in
H^{-1}(D)$, and hence for any $f\in H^{-\beta}(D)$ since $H^{-\beta}(D)\subset
H^{-1}(D)$. 

(ii) For parameters $p,q$ and $\gamma$ given above, we choose $\beta=1$ such that $\gamma<\beta$, $\frac12-\frac{\beta-\gamma}d<\frac1q$, and hence the embeddings
\[
 W^{-\gamma,p}_0( D)\hookrightarrow H^{-\beta}_0( D), \quad 
H^{\beta}( G)\hookrightarrow W^{\gamma,q}( G)
\]
are compact according to the Kondrachov compact embedding theorem (cf.
\cite{AF03}). Combining with the result in (i) yields that $H_k$ is compact from
$ W^{-\gamma,p}_0( D)$ to $ W^{\gamma,q}( G)$. 

(iii) Note that $\rho \in W^{\frac{m_\rho-d}2-\epsilon,p'}$ for any $\epsilon>0$ and $p'>1$ according to Lemma \ref{lm:iso}.
Then for any $\gamma\in(\frac{d-m_\rho}2,(\frac1q-\frac12)d+1)$, 
there exist $\epsilon>0$ and $p'>1$ such that $\frac{m_\rho-d}2-\epsilon>-\gamma$ and $\frac1{p'}-\frac{\frac{m_\rho-d}2-\epsilon+\gamma}d<\frac1{\tilde p}$ with $\tilde p=\frac p{2-p}$ and $p$ satisfying $\frac1p+\frac1q=1$, which leads to
\[
W_0^{\frac{m_\rho-d}2-\epsilon,p'}(D_\rho)\hookrightarrow W_0^{-\gamma,\tilde p}(D_\rho)
\]
according to the Kondrachov compact embedding theorem, and hence $\rho\in W_0^{-\gamma,\tilde p}(D_\rho)$.
It then follows from  \cite[Lemma 2]{LPS08} that
$\rho v\in W^{-\gamma,p}_0(D_\rho)$ for any $v\in
W^{\gamma,q}( G)$ with
\begin{align}\label{eq:rhov}
\|\rho v\|_{ W^{-\gamma,p}}\lesssim\|\rho \|_{W^{-\gamma,\tilde p}}\|v\|_{
W^{\gamma,q}}.
\end{align}
Consequently, for any $v\in W^{\gamma,q}( G)$, we have $K_kv=H_k(\rho v)\in
W^{\gamma,q}( G)$, which implies that $K_k:W^{\gamma,q}(G)\to W^{\gamma,q}(G)$ is compact according to (ii).
\end{proof}

Before showing the well-posedness of the Lippmann--Schwinger equation \eqref{eq:LS}, we present the unique continuation principle which ensures the uniqueness of the solution of \eqref{eq:LS}. 

\begin{theorem}\label{tm:continuation}
Let $\rho$ satisfy Assumption \ref{as:rho}. If $u\in W_{comp}^{\gamma,q}(\mathbb R^d)$ with $\gamma\in(0,(\frac1q-\frac12)\frac d2+\frac12)$ and $q\in(2,\frac{2d}{d-2})$ is a solution of the homogeneous equation
\[
\Delta u+k^2(1+\rho)u=0
\]
in the distribution sense, then $u\equiv0$.
\end{theorem}

\begin{proof}
For any fixed $k>0$, define an auxiliary function $v(x):=e^{-{\rm i}\eta\cdot x}u(x)$ with 
\[
\eta:=(kt,0,\cdots,0,{\rm i}k\sqrt{t^2+1})\in\mathbb C^d,\quad t>1
\]
such that $\eta\cdot\eta=-k^2$ and $\lim_{t\to\infty}|\eta|=\infty$, which satisfies
\[
(\Delta+2{\rm i}\eta\cdot\nabla)v=-k^2\rho v.
\]
The equation above is equivalent to 
\[
v=G_\eta(\rho v),
\]
where $v\in W_{comp}^{\gamma,q}(\mathbb R^d)$ and the operator $G_\eta$ is defined by
\[
G_\eta(f)(x):=\mathcal{F}^{-1}\left[\frac{k^2}{|\xi|^2+2\eta\cdot\xi}\hat f\right](x)
\]
with $\xi=(\xi_1,\cdots,\xi_d)^\top\in\mathbb R^d$.

We first give the estimate of the operator $G_\eta$. Let $D\subset\mathbb R^d$ be a bounded domain containing the supports of both $u$ and $\rho$. For any $f,g\in C_0^{\infty}(D)$, we still denote the zero extensions of $f$ and $g$ in $\mathbb R^d$ by $f$ and $g$, respectively.  For any $s\in[0,\frac12]$, by denoting $\xi^-:=(\xi_1,\cdots,\xi_{d-1})^\top\in\mathbb R^{d-1}$ and $\xi^{--}:=(\xi_2,\cdots,\xi_{d-1})^\top\in\mathbb R^{d-2}$ with $\xi^{--}=0$ if $d=2$, we get
\begin{align*}
\langle G_\eta f,g\rangle=&\langle\widehat{G_\eta f},\hat g\rangle=\int_{\mathbb R^d}\frac{k^2}{|\xi|^2+2\eta\cdot\xi}\hat f(\xi)\overline{\hat g(\xi)}d\xi\\
=&\int_{\mathbb R^d}\frac{k^2}{|\xi|^2+2kt\xi_1+2{\rm i}k\sqrt{t^2+1}\xi_d}\hat f(\xi)\overline{\hat g(\xi)}d\xi\\
=&\int_{\mathbb R^d}\frac{k^2}{(\xi_1+kt)^2-k^2t^2+|\xi^{--}|^2+\xi_d^2+2{\rm i}k\sqrt{t^2+1}\xi_d}\hat f(\xi)\overline{\hat g(\xi)}d\xi\\
=&\int_{\mathbb R^d}\frac{k^2}{|\xi^{-}|^2-k^2t^2+\xi_d^2+2{\rm i}k\sqrt{t^2+1}\xi_d}\hat f(\xi)\overline{\hat g(\xi)}d\xi\\
=&\int_{\Omega_{\rm I}}\frac{k^2(1+|\xi|^2)^s}{|\xi^{-}|^2-k^2t^2+\xi_d^2+2{\rm i}k\sqrt{t^2+1}\xi_d}\widehat{\mathcal{J}^{-s}f}(\xi)\overline{\widehat{\mathcal{J}^{-s}g}(\xi)}d\xi\\
&+\int_{\Omega_{\rm II}}\frac{k^2(1+|\xi|^2)^s}{|\xi^{-}|^2-k^2t^2+\xi_d^2+2{\rm i}k\sqrt{t^2+1}\xi_d}\widehat{\mathcal{J}^{-s}f}(\xi)\overline{\widehat{\mathcal{J}^{-s}g}(\xi)}d\xi\\
=&:{\rm I}+{\rm II}
\end{align*}
with 
\[
\Omega_{\rm I}:=\left\{\xi: ||\xi^-|-kt|>\frac{kt}2\right\}=\left\{\xi: |\xi^-|>\frac{3kt}2\text{ or }|\xi^-|<\frac{kt}2\right\}
\] 
and 
\[
\Omega_{\rm II}:=\left\{\xi: ||\xi^-|-kt|<\frac{kt}2\right\}=\left\{\xi:\frac{kt}2<|\xi^-|<\frac{3kt}2\right\}, 
\]
where the transformation of variables $(\xi_1+kt,\xi_2,\cdots,\xi_d)^\top\mapsto(\xi_1,\xi_2,\cdots,\xi_d)^\top$ and the fact $\hat f(\xi_1-kt,\xi_2,\cdots,\xi_d)=e^{-{\rm i}kt\xi_1}\hat f(\xi_1,\xi_2,\cdots,\xi_d)$ are used.

The first term ${\rm I}$ satisfies
\begin{align}\label{eq:I}
|{\rm I}|\le&\int_{\Omega_{\rm I}}\frac{k^2(1+|\xi|^2)^s}{\left[(|\xi^-|^2-k^2t^2+\xi_d^2)^2+4k^2(t^2+1)\xi_d^2\right]^{\frac12}}|\widehat{\mathcal{J}^{-s}f}||\widehat{\mathcal{J}^{-s}g}|d\xi\notag\\
=&\int_{\Omega_{\rm I}}\frac{k^2(1+|\xi|^2)^s}{\left[((|\xi^-|-kt)^2+\xi_d^2)((|\xi^-|+kt)^2+\xi_d^2)+4k^2\xi_d^2\right]^{\frac12}}|\widehat{\mathcal{J}^{-s}f}||\widehat{\mathcal{J}^{-s}g}|d\xi\notag\\
\le&\int_{\Omega_{\rm I}}\frac{k^2(1+|\xi|^2)^s}{||\xi^-|-kt|((|\xi^-|+kt)^2+\xi_d^2)^{\frac12}}|\widehat{\mathcal{J}^{-s}f}||\widehat{\mathcal{J}^{-s}g}|d\xi\notag\\
\le&\frac{2k}{t}\bigg[\int_{\{\xi:|\xi^-|>\frac{3kt}2\}}\frac{(1+|\xi|^2)^s}{((|\xi^-|+kt)^2+\xi_d^2)^{\frac12}}|\widehat{\mathcal{J}^{-s}f}||\widehat{\mathcal{J}^{-s}g}|d\xi\notag\\
&+\int_{\{\xi:|\xi^-|<\frac{kt}2,|\xi_d|<\frac{kt}2\}}\frac{(1+|\xi|^2)^s}{((|\xi^-|+kt)^2+\xi_d^2)^{\frac12}}|\widehat{\mathcal{J}^{-s}f}||\widehat{\mathcal{J}^{-s}g}|d\xi\notag\\
&+\int_{\{\xi:|\xi^-|<\frac{kt}2,|\xi_d|>\frac{kt}2\}}\frac{(1+|\xi|^2)^s}{((|\xi^-|+kt)^2+\xi_d^2)^{\frac12}}|\widehat{\mathcal{J}^{-s}f}||\widehat{\mathcal{J}^{-s}g}|d\xi\bigg]\notag\\
=&:\frac{2k}{t}\left[{\rm I}_1+{\rm I}_2+{\rm I}_3\right],
\end{align}
where in the third step we use the fact
\begin{align*}
&(|\xi^-|^2-k^2t^2+\xi_d^2)^2+4k^2(t^2+1)\xi_d^2\\
=&\left(|\xi^-|^2+\xi_d^2+k^2t^2\right)^2-4k^2t^2|\xi^-|^2+4k^2\xi_d^2\\
=&\left[(|\xi^-|^2-kt)^2+\xi_d^2\right]\left[(|\xi^-|^2+kt)^2+\xi_d^2\right]+4k^2\xi_d^2.
\end{align*}
For sufficiently large $t>0$, the following estimates hold: 
\[
{\rm I}_1\lesssim\int_{\{\xi:|\xi|>\frac{3kt}2\}}\frac{1}{|\xi|^{1-2s}}|\widehat{\mathcal{J}^{-s}f}||\widehat{\mathcal{J}^{-s}g}|d\xi
\lesssim\frac1{(kt)^{1-2s}}\|f\|_{H^{-s}(D)}\|g\|_{H^{-s}(D)},
\]
\[
{\rm I}_2\lesssim\int_{\{\xi:|\xi|<kt\}}\frac{(1+|\xi|^2)^s}{kt}|\widehat{\mathcal{J}^{-s}f}||\widehat{\mathcal{J}^{-s}g}|d\xi
\lesssim\frac1{(kt)^{1-2s}}\|f\|_{H^{-s}(D)}\|g\|_{H^{-s}(D)}
\]
and
\begin{align*}
{\rm I}_3\lesssim&\int_{\{\xi:|\xi^-|<\frac{kt}2,|\xi_d|>\frac{kt}2\}}\frac{(1+|\xi^-|^2+\xi_d^2)^s}{|\xi_d|}|\widehat{\mathcal{J}^{-s}f}||\widehat{\mathcal{J}^{-s}g}|d\xi\\
\lesssim&\int_{\{\xi:|\xi^-|<\frac{kt}2,|\xi_d|>\frac{kt}2\}}\left(\frac{|\xi^-|^{2s}}{|\xi_d|}+\frac1{|\xi_d|^{1-2s}}\right)|\widehat{\mathcal{J}^{-s}f}||\widehat{\mathcal{J}^{-s}g}|d\xi\\
\lesssim&\frac1{(kt)^{1-2s}}\|f\|_{H^{-s}(D)}\|g\|_{H^{-s}(D)},
\end{align*}
which, together with \eqref{eq:I}, lead to
\begin{align}\label{eq:estiI}
|{\rm I}|\lesssim\frac{k^{2s}}{t^{2-2s}}\|f\|_{H^{-s}(D)}\|g\|_{H^{-s}(D)}.
\end{align}

For term ${\rm II}$, a simple calculation yields 
\begin{align}\label{eq:II}
{\rm II}=&\int_{\{\xi:\frac{kt}2<|\xi^-|<\frac{3kt}2,|\xi_d|>\frac{kt}2\}}\frac{k^2(1+|\xi|^2)^s\widehat{\mathcal{J}^{-s}f}(\xi)\overline{\widehat{\mathcal{J}^{-s}g}(\xi)}}{|\xi^{-}|^2-k^2t^2+\xi_d^2+2{\rm i}k\sqrt{t^2+1}\xi_d}d\xi\notag\\
&+\int_{\{\xi:\frac{kt}2<|\xi^-|<\frac{3kt}2,|\xi_d|<\frac{kt}2\}}\frac{k^2(1+|\xi|^2)^s\widehat{\mathcal{J}^{-s}f}(\xi)\overline{\widehat{\mathcal{J}^{-s}g}(\xi)}}{|\xi^{-}|^2-k^2t^2+\xi_d^2+2{\rm i}k\sqrt{t^2+1}\xi_d}d\xi\notag\\
=&:{\rm II}_1+{\rm II}_2,
\end{align}
where ${\rm II}_1$ satisfies
\begin{align}\label{eq:II1}
|{\rm II}_1|\le&\int_{\{\xi:\frac{kt}2<|\xi^-|<\frac{3kt}2,|\xi_d|>\frac{kt}2\}}\frac{k^2(1+|\xi|^2)^s\left|\widehat{\mathcal{J}^{-s}f}\right|\left|\widehat{\mathcal{J}^{-s}g}\right|}{\left[(|\xi^-|^2-k^2t^2+\xi_d^2)^2+4k^2(t^2+1)\xi_d^2\right]^{\frac12}}d\xi\notag\\
\lesssim&\int_{\{\xi:\frac{kt}2<|\xi^-|<\frac{3kt}2,|\xi_d|>\frac{kt}2\}}\left(\frac{k^2|\xi^-|^{2s}}{kt|\xi_d|}+\frac{k^2}{kt|\xi_d|^{1-2s}}\right)\left|\widehat{\mathcal{J}^{-s}f}\right|\left|\widehat{\mathcal{J}^{-s}g}\right|d\xi\notag\\
\lesssim&\frac{k^{2s}}{t^{2-2s}}\|f\|_{H^{-s}(D)}\|g\|_{H^{-s}(D)}.
\end{align}

It then suffices to estimate ${\rm II}_2$.
Define a function
\begin{align*}
m_t(\xi):=&\frac{k^2}{|\xi^{-}|^2-k^2t^2+\xi_d^2+2{\rm i}k\sqrt{t^2+1}\xi_d}
\end{align*}
and the transformation of variables $\xi\mapsto\xi^*=(\xi',-\xi_d)$ with
\begin{equation*}
\xi'=\left(\frac{2kt}{|\xi^-|}-1\right)\xi^-
\end{equation*} 
and the Jacobian 
\[
J_t(\xi)=\left|\det\frac{\partial \xi^*}{\partial \xi}\right|=\left(\frac{2kt}{|\xi^-|}-1\right)^{d-2}
\]
such that $|\xi'|=2kt-|\xi^-|$. Clearly, the transformation maps the subdomain 
\[
\Omega_1:=\{\xi:\frac{kt}2<|\xi^-|<kt,|\xi_d|<\frac{kt}2\}
\]
to the subdomain 
\[
\Omega_2:=\{\xi:kt<|\xi^-|<\frac{3kt}2,|\xi_d|<\frac{kt}2\}.
\]
Hence, ${\rm II}_2$ satisfies
\begin{align*}
{\rm II}_2=&\int_{\Omega_1\cup\Omega_2}m_t(\xi)(1+|\xi|^2)^s\widehat{\mathcal{J}^{-s}f}(\xi)\overline{\widehat{\mathcal{J}^{-s}g}(\xi)}d\xi\\
=&\int_{\Omega_2}\Big[m_t(\xi)(1+|\xi|^2)^s\widehat{\mathcal{J}^{-s}f}(\xi)\overline{\widehat{\mathcal{J}^{-s}g}(\xi)}\\
&+m_t(\xi^*)(1+|\xi^*|^2)^s\widehat{\mathcal{J}^{-s}f}(\xi^*)\overline{\widehat{\mathcal{J}^{-s}g}(\xi^*)}J_t(\xi)\Big]d\xi\\
=&\int_{\Omega_2}\left[m_t(\xi)+m_t(\xi^*)J_t(\xi)\right](1+|\xi|^2)^s\widehat{\mathcal{J}^{-s}f}(\xi)\overline{\widehat{\mathcal{J}^{-s}g}(\xi)}d\xi\\
&+\int_{\Omega_{2}}m_t(\xi^*)J_t(\xi)\left[(1+|\xi^*|^2)^s-(1+|\xi|^2)^s\right]
\widehat{\mathcal
J^{-s}f}(\xi)\overline{\widehat{\mathcal J^{-s}g}(\xi)}d\xi\\\notag
&+\int_{\Omega_{2}}m_t(\xi^*)J_t(\xi)(1+|\xi^*|^2)^s\big[\widehat{\mathcal
J^{-s} f} (\xi^*)-\widehat{\mathcal
J^{-s}f}(\xi)\big]\overline{\widehat{\mathcal J^{-s}g} (\xi)}d\xi\\\notag
&+\int_{\Omega_{2}}m_t(\xi^*)J_t(\xi)(1+|\xi^*|^2)^s\widehat{\mathcal J^{-s}f} (\xi^*)\big[\overline{\widehat{\mathcal
J^{-s} g } (\xi^*)}-\overline{\widehat{\mathcal
J^{-s}g}(\xi)}\big]d\xi\notag\\
=&:{\rm II}_{21}+{\rm II}_{22}+{\rm II}_{23}+{\rm II}_{24}.
\end{align*}
For any $\xi\in\Omega_2$, we define the function
\[
h(\xi_d^2):=|m_t(\xi)+m_t(\xi^*)J_t(\xi)|.
\]
If $d=2$, it can be easily shown that 
\begin{align*}
h(\xi_d^2):=&\bigg|\frac{k^2}{|\xi^{-}|^2-k^2t^2+\xi_d^2+2{\rm i}k\sqrt{t^2+1}\xi_d}
+\frac{k^2}{|\xi'|^2-k^2t^2+\xi_d^2-2{\rm i}k\sqrt{t^2+1}\xi_d}\bigg|\\
=&\frac{k^2(|\xi^-|^2+|\xi'|^2-2k^2t^2+2\xi_d^2)}{\left[(|\xi^{-}|^2-k^2t^2+\xi_d^2)^2+4k^2(t^2+1)\xi_d^2\right]^{\frac12}\left[(|\xi'|^2-k^2t^2+\xi_d^2)^2+4k^2(t^2+1)\xi_d^2\right]^{\frac12}}
\end{align*}
is decreasing with respect to $\xi_d^2\in[0,\frac{k^2t^2}4)$ and hence
\begin{align*}
h(\xi_d^2)\le h(0)=\frac{2k^2}{(|\xi^-|+kt)(3kt-|\xi^-|)}\lesssim\frac1{t^2}.
\end{align*}
If $d=3$, similarly, we have 
\begin{align*}
h(\xi_d^2)\le& h(0)=k^2\bigg|\frac{1}{|\xi^{-}|^2-k^2t^2}
+\frac{(\frac{2kt}{|\xi^-|}-1)}{|\xi'|^2-k^2t^2}\bigg|\\
=&\frac{k^2}{|\xi^-|}\frac{2kt}{(|\xi^-|+kt)(3kt-|\xi^-|)}\lesssim\frac1{t^2}.
\end{align*}
As a result, we obtain 
\begin{align}\label{eq:II21}
|{\rm II}_{21}|\lesssim\frac1{t^2}\int_{\Omega_2}(1+|\xi|^2)^s\left|\widehat{\mathcal{J}^{-s}f}\right|\left|\widehat{\mathcal{J}^{-s}g}\right|d\xi\lesssim\frac{k^{2s}}{t^{2-2s}}\|f\|_{H^{-s}(D)}\|g\|_{H^{-s}(D)}.
\end{align}
By the mean value theorem, similar to the estimate of $h(\xi_d^2)$ above, we get for some $\theta\in(0,1)$ that 
\begin{align*}
&\left|m_t(\xi^*)J_t(\xi)\left[(1+|\xi^*|^2)^s-(1+|\xi|^2)^s\right]\right|\\
=&\left|m_t(\xi^*)J_t(\xi)s\left(1+\theta|\xi^*|^2+(1-\theta)|\xi|^2\right)^{s-1}(|\xi^*|^2-|\xi|^2)\right|\\
=&\left|\frac{k^2(\frac{2kt}{|\xi^-|}-1)^{d-2}(|\xi'|^2-|\xi^-|^2)}{|\xi'|^2-k^2t^2+\xi_d^2-2{\rm i}k\sqrt{t^2+1}\xi_d}\right|s\left(1+\theta|\xi^*|^2+(1-\theta)|\xi|^2\right)^{s-1}
\lesssim\frac{k^{2s}}{t^{2-2s}},
\end{align*}
which leads to
\begin{align}\label{eq:II22}
|{\rm II}_{22}|\lesssim\frac{k^{2s}}{t^{2-2s}}\|f\|_{H^{-s}(D)}\|g\|_{H^{-s}(D)}.
\end{align}
To estimate ${\rm II}_{23}$ and ${\rm II}_{24}$, we employ the following
characterization of $W^{1,p}({\mathbb R^d})$ introduced in \cite{PH}.

\begin{lemma}\label{HLL}
For $1<p\leq\infty$, the function $u\in W^{1,p}({\mathbb R^d})$ if and only if
there exist $g\in L^p({\mathbb R^d})$ and $C>0$ such that
\begin{eqnarray*}\label{x19}
|u(x)-u(y)|\leq C|x-y|(g(x)+g(y)).
\end{eqnarray*}
Moreover, we can choose $g=M(|\nabla u|)$, where $M$ is defined by
\begin{eqnarray*}\label{x20}
M(f)(x)=\sup_{r>0}\frac{1}{|B(x,r)|}\int_{B(x,r)}|f(y)|dy
\end{eqnarray*}
and is called the Hardy--Littlewood maximal function of $f$.
\end{lemma}

For $f,g\in C_0^{\infty}(D)$, we have $\widehat{\mathcal J^{-s}f},\widehat{\mathcal J^{-s}g}\in\mathcal{S}(\mathbb R^d)\subset H^1(\mathbb R^d)$. An application of Lemma \ref{HLL} gives
\begin{eqnarray*}
\left|\widehat{\mathcal
J^{-s}f}(\xi^*)-\widehat{\mathcal J^{-s}f}(\xi)\right|
\lesssim\big||\xi^*|-|\xi|\big|\big[M(|\nabla \widehat{\mathcal
J^{-s}f}|)(\xi^*)+M(|\nabla \widehat{\mathcal J^{-s}f}|)(\xi)\big],
\end{eqnarray*} 
where $M(|\nabla \widehat{\mathcal J^{-s}f}|)$ satisfies
\begin{eqnarray*}
\|M(|\nabla \widehat{\mathcal J^{-s}f}|)\|_{L^2(\mathbb R^d)}&\lesssim& \|\nabla
\widehat{\mathcal J^{-s}f}\|_{L^2(\mathbb R^d)}\lesssim
\|(I-\Delta)^{\frac12}\widehat{\mathcal J^{-s}f}\|_{L^2(\mathbb R^d)}\notag\\
&=&\|(I-\Delta)^{\frac12}(1+|\cdot|^2)^{-\frac
s2}\hat{f}(\cdot)\|_{L^2(\mathbb R^d)}\notag\\
&=&\|(I+|\cdot|)^{\frac12}(1-\Delta)^{-\frac
s2}f(\cdot)\|_{L^2(\mathbb R^d)}\notag\\
&\lesssim&\|f\|_{H^{-s}(D)}
\end{eqnarray*}
according to \cite[Theorem 2.1]{LP}, and the same for $g$.
The above estimates then lead to
\begin{eqnarray*}\label{x24}
|{\rm II}_{23}|&\lesssim& \frac{k^{1+2s}}{t^{1-2s}}\int_{\Omega_{2}}\big[M(|\nabla
\widehat{\mathcal J^{-s}f_1}|)(\xi^*)+M(|\nabla
\widehat{\mathcal
J^{-s}f}|)(\xi)\big]|\widehat{\mathcal J^{-s}g}(\xi)|d\xi\\
&\lesssim& \frac{k^{1+2s}}{t^{1-2s}}\|f\|_{H^{-s}(D)}\|g\|_{H^{-s}(D)}
\end{eqnarray*}
and 
\begin{eqnarray*}
|{\rm II}_{24}|\lesssim\frac{k^{1+2s}}{t^{1-2s}}\|f\|_{H^{-s}(D)}\|g\|_{H^{-s}(D)},
\end{eqnarray*}
which, together with \eqref{eq:II21} and \eqref{eq:II22}, yield
\begin{align}\label{eq:II2}
|{\rm II}_2|\lesssim\frac{(1+k)k^{2s}}{t^{1-2s}}\|f\|_{H^{-s}(D)}\|g\|_{H^{-s}(D)}.
\end{align}

We conclude from \eqref{eq:estiI}--\eqref{eq:II1} and \eqref{eq:II2} that
\[
|\langle G_\eta f,g\rangle|\lesssim \frac{(1+k)k^{2s}}{t^{1-2s}}\|f\|_{H^{-s}(D)}\|g\|_{H^{-s}(D)}
\]
for any $f,g\in C_0^\infty(D)$, which can be easily extended to $f,g\in H^{-s}(D)$ by taking the procedure used in the estimate of \eqref{eq:Hk}. Hence, we get
\[
\|G_\eta\|_{\mathcal{L}(H^{-s}(D),H^s(D))}\lesssim\frac{(1+k)k^{2s}}{t^{1-2s}},
\]
which, together with \cite[Proposition 2]{LPS08}, leads to
\[
\|G_\eta\|_{\mathcal{L}(W^{-\gamma,p}(D),W^{\gamma,q}(D))}\lesssim\frac{(1+k)^\theta k^{2s\theta}}{t^{(1-2s)\theta}},
\]
where $p$ satisfies $\frac1p+\frac1q=1$, $\theta=(\frac1q-\frac12)d+1$ and $\gamma=\theta s\in(0,(\frac1q-\frac12)\frac d2+\frac12)$. Utilizing \eqref{eq:rhov} gives 
\[
\|v\|_{W^{\gamma,q}(D)}=\|G_\eta(\rho v)\|_{W^{\gamma,q}(D)}\lesssim\frac{(1+k)^\theta k^{2s\theta}}{t^{(1-2s)\theta}}\|\rho \|_{W^{-\gamma,\tilde p}(D)}\|v\|_{W^{\gamma,q}(D)},
\]
where $\tilde p=\frac p{2-p}$. The proof is completed by letting $t\to\infty$.
\end{proof}

\begin{theorem}\label{tm:LS}
Let $\rho$ satisfy Assumption \ref{as:rho}. Then the Lippmann--Schwinger equation
\eqref{eq:LS} admits a unique solution $ u\in W^{\gamma,q}_{loc}(\mathbb R^d)$ almost surely with $\gamma\in(\frac{d-m}2,(\frac1q-\frac12)\frac d2+\frac12)$ and $q\in(2,\frac{2d}{3d-2-2m})$, where
\begin{itemize}
\item[(i)] $m=m_\rho\wedge m_f$ if the condition (i) in Assumption \ref{as:f} holds
\end{itemize}
or
\begin{itemize}
\item[(ii)] $m=m_\rho$ if the condition (ii) in Assumption \ref{as:f} holds.
\end{itemize}
\end{theorem}

\begin{proof}
Let $ G\subset\mathbb{R}^d$ be any bounded set with a locally Lipschitz
boundary. Based on the definition of the operator $K_k$, the Lippmann--Schwinger
equation can be written in the form
\begin{align}\label{eq:LSoperators}
(I-k^2K_k) u= -H_kf,
\end{align}
where the operator $I-k^2K_k: W^{\gamma,q}( G)\to W^{\gamma,q}( G)$ is Fredholm according to Lemma \ref{lm:operators}. 
It follows from the Fredholm alternative theorem that \eqref{eq:LSoperators} has
a unique solution $ u\in W^{\gamma,q}( G)$ if 
\begin{align}\label{eq:LShomo}
(I-k^2K_k) u=0
\end{align}
has only the trivial solution $ u\equiv0$, which has been proved in Theorem \ref{tm:continuation}.

Next is to show $H_kf\in W^{\gamma,q}( G)$. We consider the following two
cases:

If the condition (i) in Assumption \ref{as:f} holds, for any $q,\gamma$ given above and $p$ satisfying $\frac1p+\frac1q=1$, there
exist $\epsilon>0$ and $p'>1$ such that 
\[
f\in W_0^{\frac{m_f-d}2-\epsilon,p'}(D_f)\hookrightarrow  W^{-\gamma,p}_0(D_f)
\] 
and hence $H_kf\in W^{\gamma,q}( G)$ according to Lemma \ref{lm:operators}.

If the condition (ii) in Assumption \ref{as:f} holds, it holds
$H_kf=\Phi_d(\cdot,y,k)\in W^{1,p'}( G)$ for any $p'\in(1,3-\frac d2)$ according
to \cite[Lemma 3.1]{LW}. Then there exists $p'=3-\frac d2-\epsilon$ for a
sufficiently small $\epsilon>0$ satisfying
$\frac1{p'}-\frac{1-\gamma}d<\frac1q$, such that 
\[
W^{1,p'}( G)\hookrightarrow  W^{\gamma,q}( G),
\]
and hence $H_kf=\Phi_d(\cdot,y,k)\in W^{\gamma,q}( G)$. 
\end{proof}

\subsection{Well-posedness}

Now we present the well-posedness on the solution of
\eqref{eq:H}--\eqref{eq:RCH} in the distribution sense by showing the
equivalence to the Lippmann--Schwinger
equation.

\begin{theorem}\label{tm:wellposedH}
Let $\rho$ satisfy Assumption \ref{as:rho}. The acoustic scattering problem
\eqref{eq:H}--\eqref{eq:RCH} is well-defined in the
distribution sense, and admits a unique solution $u\in W^{\gamma,q}_{loc}(\mathbb R^d)$ almost surely with $\gamma\in(\frac{d-m}2,(\frac1q-\frac12)\frac d2+\frac12)$ and $q\in(2,\frac{2d}{3d-2-2m})$, where
\begin{itemize}
\item[(i)] $m=m_\rho\wedge m_f$ if the condition (i) in Assumption \ref{as:f} holds
\end{itemize}
or
\begin{itemize}
\item[(ii)] $m=m_\rho$ if the condition (ii) in Assumption \ref{as:f} holds.
\end{itemize}
\end{theorem}

\begin{proof}
First we show the existence of the solution of \eqref{eq:H}--\eqref{eq:RCH}.
Specifically, we show that the solution of the Lippmann--Schwinger equation
\eqref{eq:LS} is also a solution of \eqref{eq:H}--\eqref{eq:RCH} in the
distribution sense. Suppose that $ u^*\in
W^{\gamma,q}_{loc}(\mathbb R^d)$ is the solution of \eqref{eq:LS} and satisfies
\[
 u^*(x)-k^2\int_{\mathbb{R}^d} \Phi_d(x,z,k)\rho(z) u^*(z)dz= -\int_{\mathbb R^d}\Phi_d(x,z,k)f(z)dz,\quad x\in\mathbb{R}^d.
\]

Note that the Green tensor $ \Phi_d$ is the fundamental solution for the operator $\Delta+k^2 I$:
\[
(\Delta+k^2) \Phi_d(\cdot,z,k)=-\delta(\cdot-z),
\]
where the Dirac delta function $\delta$ is a distribution, i.e., $\delta\in\mathcal{D}'$. It indicates that, for any $\psi\in\mathcal{D}$, 
\[
\langle(\Delta+k^2) \Phi_d(\cdot,z,k),\psi\rangle=-\langle\delta(\cdot-z),\psi\rangle=-\overline{\psi(y)}.
\]
We obtain for any $\psi\in\mathcal{D}$ that 
\begin{eqnarray*}
&&\langle\Delta u^*+k^2 u^*+k^2\rho u^*,\psi\rangle\\
&=&k^2\left\langle\int_{\mathbb{R}^d}\left(\Delta+k^2\right)
\Phi_d(\cdot,z,k)\rho(z) u^*(z)dz,\psi\right\rangle\\
&&\quad -\left\langle\int_{\mathbb{R}^d}\left(\Delta+k^2\right)
\Phi_d(x,z,k)f(z)dz,\psi\right\rangle
+k^2\langle \rho  u^*,\psi\rangle\\
&=&k^2\int_{\mathbb{R}^d}\rho(z) u^*(z)\left\langle\left(\Delta+k^2\right)
\Phi_d(\cdot,z,k),\psi\right\rangle dz\\
&&\quad -\int_{\mathbb R^d}f(z)\left\langle\left(\Delta+k^2\right)
\Phi_d(\cdot,y,k),\psi\right\rangle dz +k^2\langle \rho  u^*,\psi\rangle\\
&=&-k^2\int_{\mathbb{R}^d}\rho(z) u^*(z)\overline{\psi(z)}dz+\int_{\mathbb
R^d}f(z)\overline{\psi(z)}dz+k^2\langle \rho  u^*,\psi\rangle\\
&=&\langle f,\psi\rangle.
\end{eqnarray*}
Hence, $ u^*\in W^{\gamma,q}_{loc}(\mathbb R^d)$ is also a solution of
\eqref{eq:H}--\eqref{eq:RCH} in the
distribution sense, which shows the existence of the solution of
\eqref{eq:H}--\eqref{eq:RCH}
according to Theorem \ref{tm:LS}.

The uniqueness of the solution of \eqref{eq:H}--\eqref{eq:RCH} can be proved by
using the
same procedure as that of the Lippmann--Schwinger equation. Let $ u_0$ be any
solution of \eqref{eq:H} with $f=0$ in the distribution sense. It then suffices
to show that $u_0$ is also a solution of \eqref{eq:LShomo} with $f=0$, i.e.,
$u_0\equiv0$. 
In fact, $ u_0$ satisfies
\[
\Delta u_0+k^2 u_0=-k^2\rho  u_0
\]
in the distribution sense, where $\rho \in W^{-\gamma,\tilde p}(D_\rho)$, $ u_0\in W^{\gamma,q}_{loc}(\mathbb R^d)$ and hence $\rho  u_0\in W_0^{-\gamma,p}(D_\rho)$ with $\tilde p=\frac p{2-p}$ according to the proof of Lemma \ref{lm:operators}.
Let $B_r$ be an open ball with radius $r$ large enough such that $D_\rho\subset B_r$. 

Moreover, it has been shown in Theorem \ref{tm:LS} that $\Phi_d(\cdot,y,k)\in W^{1,p'}(B_r)\hookrightarrow W^{\gamma,q}(B_r)$ with $p'=3-\frac d2-\epsilon$ for a sufficiently small $\epsilon>0$.
It then indicates that
\begin{eqnarray}\label{b12}
&&\int_{B_r} \Phi_d(x,z,k)\left[\Delta u_0(z)+k^2 u_0(z)\right]dz\notag\\
&=&-k^2\int_{B_r} \Phi_d(x,z,k)\rho(z) u_0(z)dz.
\end{eqnarray}
Define the operator $T$ by
\[
(T\psi)(x):=\int_{B_r} \Phi_d(x,z,k)\left[\Delta \psi(z)+k^2 \psi(z)\right]dz
\]
for $\psi\in\mathcal{D}$. By the similar arguments
as those in the proof of \cite[Lemma 4.3]{LHL}, we obtain
\begin{eqnarray*}\label{b13}
(T\psi)(x)=-\psi(x)+\int_{\partial
B_r}\left[ \Phi_d(x,z,k)\partial_{\nu}\psi(z)-\partial_\nu\Phi_d(x,z,k)\psi(z)\right]ds(z),
\end{eqnarray*}
where $\nu$ is the unit outward normal vector on the boundary $\partial B_r$.
Then \eqref{b12} turns to be
\begin{eqnarray*}
&&u_0(x)-\int_{\partial B_r}\Big[ \Phi_d(x,z,k)\partial_{\nu} u_0(z)-\partial_{\nu}\Phi_d(x,z,k) u_0(z)\Big]ds(z)\\
&=&k^2\int_{B_r} \Phi_d(x,z,k)\rho(z) u_0(z)dz.
\end{eqnarray*}
Let $r\to\infty$ and applying the radiation condition, we get
\begin{eqnarray*}
 u_0(x)=k^2\int_{\mathbb R^d} \Phi_d(x,z,k)\rho(z) u_0(z)dz, 
\end{eqnarray*}
which implies that $ u_0$ is also a solution of the Lippmann--Schwinger
equation \eqref{eq:LS} with $f=0$, and hence $ u_0\equiv0$ according to Theorem
\ref{tm:LS}. 
\end{proof}

\section{The elastic scattering problem}

In this section, we discuss the well-posedness of the elastic wave equation
\eqref{eq:E} in the distribution sense, where the medium $\bm M$ and the source
$\bm f$ satisfy the following assumptions. 

\begin{assumption}\label{as:M}
Let the medium $\bm M=(M_{ij})_{d\times d}$ be a $\mathbb R^{d\times
d}$-valued and centered microlocally isotropic Gaussian random field of order
$-m_{\bm M}$ with $m_{\bm M}\in(d-1,d]$ in a bounded domain $D_{\bm M}\subset\mathbb R^d$. 
The principal symbol of the covariance
operator of each component $M_{ij}$ has the form $\phi_{ij}(x)|\xi|^{-m_{\bm
M}}$ with $\phi_{ij}\in C_0^\infty(D_{\bm M})$, $\phi_{ij}\ge0$ and
$i,j=1,\cdots,d$.
\end{assumption}

\begin{remark}
For a random medium $\bm M=(M_{ij})_{d\times d}$, if the components are
centered microlocally isotropic Gaussian random fields of different orders,
denoted by $-m_{ij}$, then $\bm M$ satisfies Assumption \ref{as:M} with $m_{\bm
M}:=\min_{i,j\in\{1,\cdots,d\}}m_{ij}$. Moreover, for any component $M_{ij}$
with $m_{ij}>m_{\bm M}$, it holds $\phi_{ij}\equiv0$.
\end{remark}

\begin{assumption}\label{as:f2}
Let the $\mathbb R^d$-valued source $\bm f$ satisfy one of the following assumptions:
\begin{itemize}
\item[(i)] $\bm f$ is a centered microlocally isotropic Gaussian random vector
field of order $-m_{\bm f}$ with $m_{\bm f}\in(d-1,d]$ in a bounded domain $D_{\bm f}\subset\mathbb R^d$.
The principal symbol of its covariance operator has the form $A_{\bm
f}(x)|\xi|^{-m_{\bm f}}$ with $A_{\bm f}\in C_0^\infty(D_{\bm f};\mathbb
R^{d\times d})$.

\item[(ii)] $\bm f=-\delta(\cdot-y)\bm a$ is a point source with $y\in\mathbb
R^d$ and some fixed vector $\bm a\in\mathbb R^d$.
\end{itemize} 
\end{assumption}

In the sequel, we denote by 
\[
\bm{X}:=X^d=\{\bm{g}=(g_1,\cdots,g_d)^\top:g_j\in
X~\forall\, j=1,\cdots,d\} 
\]
the Cartesian product vector space, and use
notations $\bm{W}^{r,p}:=(W^{r,p}(\mathbb{R}^d))^d$ and $\bm{H}^r:=\bm{W}^{r,2}$
for simplicity.

\subsection{The Lippmann--Schwinger equation}

Similarly, we consider the equivalent Lippmann--Schwinger equation for the
elastic wave scattering problem. Denote by $\bm{\Phi}_d(x,y,k)\in\mathbb
C^{d\times d}$ the Green tensor for the Navier equation which has the following
form:
\begin{eqnarray}\label{eq:Gtensor}
\bm{\Phi}_d(x,y,k)=\frac{1}{\mu}\Phi_d(x,y,\kappa_{\rm
s})\bm{I}+\frac{1}{\omega^2}\nabla_x\nabla_x^\top\Big[\Phi_d(x,y,
\kappa_{\rm s})-\Phi_d(x,y,\kappa_{\rm p})\Big],
\end{eqnarray}
where $\bm{I}$ is the $d\times d$ identity matrix and $\Phi_d$ is the fundamental
solution for the Helmholtz equation and is 
defined in \eqref{eq:Green}. 

Based on the Green tensor $\bm \Phi_d$ given in \eqref{eq:Gtensor}, the Lippmann--Schwinger equation has the form
\begin{eqnarray}\label{eq:LS2}
\bm{u}(x)-k^2\int_{\mathbb R^d}\bm{\Phi}_d(x,z,k)\bm M(z)\bm{u}(z)dz=-\int_{\mathbb R^d}\bm{\Phi}_d(x,z,k)\bm f(z)dz.
\end{eqnarray}
Define two operators ${\bm H}_k$ and ${\bm K}_k$ by
\begin{align*}
({\bm H}_k \bm v)(x)&=\int_{\mathbb{R}^d}\bm{\Phi}_d(x,z,k)\bm{v}(z)dz,\\
({\bm K}_k \bm v)(x)&=\int_{\mathbb{R}^d}\bm{\Phi}_d(x,z,k)\bm M(z)\bm{v}(z)dz,
\end{align*}
which have the following properties.

\begin{lemma}\label{lm:operators2}
Let $\bm M$ satisfy Assumption \ref{as:M}. 
Let $ D\subset\mathbb{R}^d$ be a bounded set and $ G\subset\mathbb{R}^d$ be a bounded set with a locally Lipschitz boundary.
\begin{itemize}
\item[(i)] The operator ${\bm H}_k:\bm{H}_0^{-\beta}( D)\to\bm{H}^\beta( G)$ is bounded for any $\beta\in(0,1]$.

\item[(ii)] The operator ${\bm H}_k:\bm{W}_0^{-\gamma,p}( D)\to\bm{W}^{\gamma,q}( G)$ is compact for any $q\in(0,\infty)$, $\gamma\in(0,(\frac1q-\frac12)d+1)$ and $p$ satisfying $\frac1p+\frac1q=1$.

\item[(iii)] The operator ${\bm K}_k:\bm{W}^{\gamma,q}( G)\to\bm{W}^{\gamma,q}(G)$ is compact almost surely for any $q\in(2,\frac{2d}{2d-2-m_{\bm M}})$ and $\gamma\in(\frac{d-m_{\bm M}}2,(\frac1q-\frac12)d+1)$.
\end{itemize}
\end{lemma}

\begin{proof}
Noting that the linear operator in \eqref{eq:E} is uniformly elliptic and $\bm
H_k$ is bounded from $\bm C^{0,\alpha}(D)$ to $\bm C^{2,\alpha}(G)$ (cf.
\cite[Theorem 6.8]{GT01}), we may obtain the result in (i) by following
essentially the same procedure as that for Lemma \ref{lm:operators}. The
details are omitted for brevity. 

The proof of (ii) can also be obtained by using the same procedure as the proof
of Lemma 3.1 and noting that the embeddings
\[
\bm W^{-\gamma,p}_0( D)\hookrightarrow\bm H^{-\beta}_0( D),\quad 
\bm H^{\beta}( G)\hookrightarrow\bm W^{\gamma,q}( G)
\]
hold by choosing $\beta=1$ such that $\gamma<\beta$ and $\frac12-\frac{\beta-\gamma}d<\frac1q$.

It then suffices to show (iii). Note that $\bm M \in (W^{\frac{m_{\bm M}-d}2-\epsilon,p'})^{d\times d}$ almost surely for any $\epsilon>0$ and $p'>1$ according to Lemma \ref{lm:iso}, and there must exist $\epsilon>0$ and $p'>1$ such that the embedding
\[
W_0^{\frac{m_{\bm M}-d}2-\epsilon,p'}(D_{\bm M})\hookrightarrow W_0^{-\gamma,\tilde p}(D_{\bm M})
\]
holds according to the Kondrachov compact embedding theorem with $\tilde p=\frac p{2-p}$.
Hence, $\bm M\in (W^{-\gamma,\tilde p}(\mathbb R^d))^{d\times d}$ and $\bm M\bm v\in \bm W^{-\gamma,p}$ almost surely for any $\bm v\in\bm W^{\gamma,q}(G)$ with
\[
\|\bm M\bm v\|_{\bm W^{-\gamma,p}}\lesssim\|\bm M\|_{(W^{-\gamma,\tilde p}(\mathbb R^d))^{d\times d}}\|\bm v\|_{\bm W^{\gamma,q}}
\]
according to \cite[Lemma 2]{LPS08}, where
\[
\|\bm M\|_{(W^{-\gamma,\tilde p})^{d\times
d}}:=\left(\sum_{i,j=1}^d\|M_{ij}\|_{W^{-\gamma,\tilde
p}(\mathbb R^d)}^2\right)^{\frac12},\quad 
\|\bm v\|_{\bm W^{\gamma,q}}:=\left(\sum_{j=1}^d\|v_j\|_{W^{\gamma,q}(\mathbb R^d)}^2\right)^\frac12
\]
for any $\bm v=(v_1,\cdots,v_d)^\top$.
As a result, for any $\bm v\in\bm W^{\gamma,q}( G)$, we get $\bm K_kv=\bm H_k(\bm M\bm v)\in\bm W^{\gamma,q}( G)$ almost surely, which completes the proof.
\end{proof}

\begin{theorem}\label{tm:LS2}
Let $\bm M$ satisfy Assumption \ref{as:M}. Then the Lippmann--Schwinger equation
\eqref{eq:LS2} admits a unique solution $\bm{u}\in\bm{W}^{\gamma,q}_{loc}$
almost surely with $\gamma\in(\frac{d-m}2,(\frac1q-\frac12)\frac d2+\frac12)$ and $q\in(2,\frac{2d}{3d-2-2m})$, where 
\begin{itemize}
\item[(i)] $m=m_{\bm M}\wedge m_{\bm f}$ if the condition (i) in Assumption \ref{as:f2} holds
\end{itemize}
or
\begin{itemize}
\item[(ii)] $m=m_{\bm M}$ if the condition (ii) in Assumption \ref{as:f2} holds.
\end{itemize}
\end{theorem}

Note that, according to the Helmholtz decomposition, the solution $\bm u$ of the homogeneous elastic wave equation with $\bm f\equiv0$ in \eqref{eq:E} can be decomposed into the compressional part $\bm u_{\rm p}$ and the shear part $\bm u_{\rm s}$ such that $\bm u=\bm u_{\rm p}+\bm u_{\rm s}$. Both $\bm u_{\rm p}$ and $\bm u_{\rm s}$ satisfy the Helmholtz equation with Sommerfeld radiation condition (cf. \cite{LL19}).
Hence, the proof of Theorem \ref{tm:LS2} can be obtained similarly by following the same procedure as the proof of Theorem \ref{tm:LS} utilizing the fact $\bm{\Phi}_d(\cdot,y,k)\in(W^{1,p'}(G))^{d\times d}$ with
$p'\in(1,3-\frac d2)$ shown in \cite[Lemma 4.1]{LHL}. The details are
omitted for brevity.

\subsection{Well-posedness}

Now we present the existence and uniqueness of the solution of
\eqref{eq:E}--\eqref{eq:RCE} in the distribution sense by utilizing the
Lippmann--Schwinger
equation for the elastic scattering problem. 

\begin{theorem}\label{tm:wellposedE}
Let $\bm M$ satisfy Assumption \ref{as:M}.  The elastic scattering problem \eqref{eq:E}--\eqref{eq:RCE} is well-defined in the distribution
sense, and admits a unique solution $\bm{u}\in\bm{W}^{\gamma,q}_{loc}$ almost
surely with $\gamma\in(\frac{d-m}2,(\frac1q-\frac12)\frac d2+\frac12)$ and $q\in(2,\frac{2d}{3d-2-2m})$, where 
\begin{itemize}
\item[(i)] $m=m_{\bm M}\wedge m_{\bm f}$ if the condition (i) in Assumption \ref{as:f2} holds
\end{itemize}
or
\begin{itemize}
\item[(ii)] $m=m_{\bm M}$ if the condition (ii) in Assumption \ref{as:f2} holds.
\end{itemize}
\end{theorem}

\begin{proof}
To show the existence of the solution, we first show that the solution of the
Lippmann--Schwinger equation \eqref{eq:LS2} is also a solution of
\eqref{eq:E}--\eqref{eq:RCE}
in the distribution sense. Suppose that
$\bm{u}^*\in\boldsymbol{W}^{\gamma,q}_{loc}$ is the solution of \eqref{eq:LS2}
and satisfies
\[
\bm{u}^*(x)-k^2\int_{\mathbb R^d}\bm{\Phi}_d(x,z,k){\bm M}(z)\bm{u}^*(z)dz=-\int_{\mathbb R^d}\bm{\Phi}_d(x,z,k)\bm f(z)dz,\quad
x\in\mathbb R^d.
\]
The Green tensor $\bm{\Phi}_d$ is the fundamental solution of the
equation
\[
(\Delta^*+k^2)\bm{\Phi}_d(\cdot,y,k)=-\delta(\cdot-y)\bm I,
\]
where $\Delta^*:=\mu\Delta+(\mu+\lambda)\nabla\nabla\cdot$.
For any $\boldsymbol{\psi}\in\boldsymbol{\mathcal{D}}$, it is easy to note that 
\[
\langle(\Delta^*+k^2)\bm{\Phi}_d(\cdot,y,k),\boldsymbol{\psi}\rangle=-\langle\delta(\cdot-y)\bm I,\boldsymbol{\psi}\rangle=-\overline{\boldsymbol{\psi}(y)}.
\]
Hence, for any $\boldsymbol{\psi}\in\boldsymbol{\mathcal{D}}$, we get
\begin{eqnarray*}
&&\langle\Delta^*\bm{u}^*+k^2(\bm I+{\bm M})\bm{u}^*,\boldsymbol{\psi}\rangle\\
&=&k^2\left\langle\int_{\mathbb{R}^d}\left(\Delta^*+k^2\right)\bm{\Phi}_d(\cdot,
z,k){\bm M}(z)\bm{u}^*(z)dz,\boldsymbol{\psi}\right\rangle\\
&&\quad -\left\langle\int_{\mathbb
R^d}\left(\Delta^*+k^2\right)\bm{\Phi}_d(\cdot,z,k)\bm
f(z)dz,\boldsymbol{\psi}\right\rangle
+k^2\langle {\bm M}\bm{u}^*,\boldsymbol{\psi}\rangle\\
&=&k^2\int_{\mathbb{R}^d}\left({\bm
M}(z)\bm{u}^*(z)\right)^\top\left\langle\left(\Delta^*+k^2\right)\bm{\Phi}
_d(\cdot,z,k),\boldsymbol{\psi}\right\rangle dz\\
&&\quad -\int_{\mathbb
R^d}\bm{f}(z)^\top\left\langle\left(\Delta^*+k^2\right)\bm{\Phi}_d(\cdot,z,k),
\boldsymbol{\psi}\right\rangle dz+k^2\langle {\bm
M}\bm{u}^*,\boldsymbol{\psi}\rangle\\
&=&-k^2\int_{\mathbb{R}^d}\left({\bm
M}(z)\bm{u}^*(z)\right)^\top\overline{\boldsymbol{\psi}(z)}dz+\int_{\mathbb
R^d}\bm{f}(z)^\top\overline{\boldsymbol{\psi}(z)}dz+k^2\langle {\bm
M}\bm{u}^*,\boldsymbol{\psi}\rangle\\
&=&\langle \bm f,\boldsymbol{\psi}\rangle.
\end{eqnarray*}
Hence, $\bm{u}^*\in\bm{W}^{\gamma,q}_{loc}$ is also a solution of
\eqref{eq:E}--\eqref{eq:RCE}
in the distribution sense, which shows the existence of the solution of
\eqref{eq:E}--\eqref{eq:RCE} according to Theorem \ref{tm:LS2}.

The uniqueness of the solution of \eqref{eq:E}--\eqref{eq:RCE} is obtained based
on the same procedure as the proof of Theorem \ref{tm:wellposedH}.
\end{proof}



\begin{thebibliography}{00}

\bibitem{AF03}
R. Adams and J. Fournier, Sobolev Spaces, 2nd ed., Pure and Applied Mathematics
140, Elsevier/Academic Press, Amsterdam, 2003.

\bibitem{BLZ20}
G. Bao, P. Li, and Y. Zhao, Stability for the inverse source problems in elastic
and electromagnetic waves, J.
Math. Pures Appl., 134 (2020), 122--178.

\bibitem{CHL19}
P. Caro, T. Helin, and M. Lassas, Inverse scattering for a random potential,
Anal. Appl. (Singap.), 17 (2019), 513--567.

\bibitem{C75}
P.-L. Chow, Perturbation methods in stochastic wave propagation, SIAM Review, 17
(1975), 57--80.

\bibitem{CK13}
D. Colton and R. Kress, Inverse Acoustic and Electromagnetic Scattering Theory,
3rd ed.,  Applied Mathematical Sciences 93, Springer, New York, 2013.


\bibitem{DF98}
R. C. Dalang and N. E. Frangos, The stochastic wave equation in two spatial
dimensions, Ann. Probab., 26 (1998), 187--212.


\bibitem{FGPS07}
J.-P. Fouque, J. Garnier, G. C. Papanicolaou, and K. S\o{}lna, Wave Propagation
and Time Reversal in Randomly Layered Media, Springer, New York, 2007.


\bibitem{GT01} 
D. Gilbarg and N. S. Trudinger, Elliptic Partial Differential Equations of
Second Order, Reprint of the 1998 edition, Classics in Mathematics,
Springer-Verlag, Berlin, 2001.
 
\bibitem{G11}
P. Grisvard, Elliptic Problems in Nonsmooth Domains, SIAM, Philadelphia, PA,
2011.

\bibitem{PH}
P. Hajlasz, Sobolev spaces on an arbitrary metric space, Potential Anal., 5 (1996), 403--415.

\bibitem{I78}
A. Ishimaru, Wave Propagation and Scattering in Random Media, Academic Press,
New York, 1978.

\bibitem{J02}
J. Jost, Partial Differential Equations, Graduate Texts in Mathematics, vol.
214, Springer-Verlag, New York, 2002.

\bibitem{K62}
J. B. Keller, Wave propagation in random media, Proc. Sympos. Appl. Math., vol.
13, American Mathematical Society, Providence, R.I., 1962, 227--246.

\bibitem{LPS08} 
M. Lassas, L. P\"{a}iv\"{a}rinta, and E. Saksman, Inverse scattering problem for
a two dimensional random potential, Commun. Math. Phys., 279 (2008), 669--703.

\bibitem{LHL}
J. Li, T. Helin, and P. Li, Inverse random source problems for
time-harmonic acoustic and elastic waves, Commun. Part. Diff. Eqs., 45 (2020), 1335--1380. 

\bibitem{LL19}
J. Li and P. Li, Inverse elastic scattering for a random source, 
SIAM J. Math. Anal., 51 (2019), 4570--4603.

\bibitem{LLW}
J. Li, P. Li, and X. Wang, Inverse elastic scattering for a random potential,
arXiv:2007.05790.


\bibitem{LSSW16}
A. Lodhia, S. Sheffield, X. Sun, and S. Watson,
Fractional Gaussian fields: a survey, Probab. Surv., 13 (2016), 1--56.

\bibitem{LW}
P. Li and X. Wang, Inverse random source scattering for the Helmholtz equation
with attenuation, SIAM J. Appl. Math., to appear. 

\bibitem{M00}
W. McLean, Strongly Elliptic Systems and Boundary Integral Equations, Cambridge
University Press, Cambridge, 2000.

\bibitem{MS99}
A. Millet and M. Sanz-Sol\'{e}, A stochastic wave equation in two space
dimension: smoothness of the law, Ann. Probab., 27 (1999), 803--844.

\bibitem{LP}
L. P\"{a}iv\"{a}rinta, Analytic methods for inverse scattering theory, In:
New Analytic and Geometric Methods in Inverse Problems, Springer Lecture
Notes, ed. K. Bingham, Y. Kurylev, E. Somersalo, Springer, New
York, 2003, 165--185.

\end{thebibliography}
\end{document}